\documentclass[centering,11pt,reqno]{amsart}

\usepackage[utf8]{inputenc}
\usepackage[T1]{fontenc}
\usepackage{amsmath,amsthm}

\usepackage{thmtools}

\usepackage{amsfonts,amssymb}
\usepackage[mathscr]{eucal}
\usepackage{url}
\usepackage{paralist}
\usepackage[colorlinks,cite color=blue,pagebackref=true,pdftex]{hyperref}
\usepackage[margin=2.5cm]{geometry}
\usepackage{tikz}
\usetikzlibrary{calc,shapes}
\usepackage{mathtools}
\usepackage{multicol}
\usepackage{ytableau}
\usepackage{blkarray}
\usepackage[capitalize]{cleveref}
 \usepackage[foot]{amsaddr}
\usepackage{tikz-cd}
\usepackage{ytableau}
\usepackage{bbm}
\usepackage{comment}
\usepackage{color}
\usepackage{todonotes}
\usepackage{accsupp}

\newtheorem{theorem}{Theorem}[section]
\newtheorem*{theorem*}{Theorem}
\newtheorem*{cor*}{Corollary}
\newtheorem*{conj*}{Conjecture}
\newtheorem*{lemma*}{Lemma}
\newtheorem*{prop*}{Proposition}
\newtheorem{lem}[theorem]{Lemma}
\newtheorem{cor}[theorem]{Corollary}
\newtheorem{prop}[theorem]{Proposition}

\newtheorem{notation}[theorem]{Notation}

\theoremstyle{definition}

\newtheorem{definition}[theorem]{Definition}

\newtheorem{construction}[theorem]{Construction}

\newcommand\note[1]{\textbf{#1}}

\renewcommand{\epsilon}{\varepsilon}
\newcommand{\R}{\mathbb{R}}

\newcommand{\conv}{\mathrm{conv}}

\def\R{\mathbb{R}}

\newcommand{\T}{\intercal}

\renewcommand{\R}{\mathbb{R}}

\renewcommand{\c}{\mathbf{c}}
\newcommand{\e}{\mathbf{e}}
\newcommand{\p}{\mathbf{p}}
\renewcommand{\r}{\mathbf{r}}

\newcommand{\n}{\mathbf{n}}
\renewcommand{\u}{\mathbf{u}}
\renewcommand{\v}{\mathbf{v}}
\newcommand{\w}{\mathbf{w}}
\newcommand{\x}{\mathbf{x}}
\newcommand{\y}{\mathbf{y}}



\title{Any Proof of Polynomial Hirsch Must be Completely Incoherent}

\author[Black]{Alexander E. Black}
\address{Department of Mathematics, Bowdoin College}
\email{a.black@bowdoin.edu}

\author[Xue]{Lei Xue}
\address{Department of Mathematics, Colby College}
\email{leixue@colby.edu}

\begin{document}

\maketitle

\begin{abstract} 
In 1992, Billera and Sturmfels introduced coherent monotone paths on polytopes as part of their description of the fiber polytope construction, and later in 1994 showed with Kapranov that these coherent monotone paths capture the topology of the space of all monotone paths, paths from a minimum to a maximum, in the directed graph of a polytope with orientation induced by a linear function. Those results motivate the following analog of the polynomial Hirsch conjecture: Does there always exist a coherent monotone path of polynomial length on a polytope for any choice of orientation induced by a linear function? We show this is not the case by exhibiting a family of polytopes and corresponding linear functions for which every coherent monotone path is exponentially long. As applications, we strengthen longstanding results pertaining to lower bounds for the shadow simplex method, geometric transversals in discrete geometry, and parametric linear optimization.
\end{abstract}

\section{Introduction}

The polynomial Hirsch conjecture is among the most notorious and well studied problems in polyhedral theory. It asks for a polynomial bound in terms of the number of facets $n$ and dimension $d$ of a polytope on the length of a shortest path between any pair of vertices in its one-skeleton, its set of vertices and edges. This problem is so well studied that to survey all partial results is well beyond the scope of this paper, and we defer to the survey of Kim and Santos \cite{HirschUpdate} and the comprehensive references to the simplex method literature found in \cite{ByTheBook}. Famously, Hirsch originally conjectured a bound of $n-d$, and this was disproven by Santos in \cite{HirschSolution}. Despite this progress, the best known lower bounds remain linear on the order of about $1.05(n-d)$. The best known upper bounds are on the order of $n^{\log(d)}$ originally due to Kalai and Kleitman in \cite{KalaiKleitman}. 

The core motivation for this problem comes from the simplex method for linear programming. Namely, the simplex method solves a linear program $\max(\mathbf{c}^{\intercal} \mathbf{x})$ such that $A \mathbf{x} \leq \mathbf{b}$ by tracing a path in the one-skeleton of the polytope such that at each step the linear objective function increases. It is a major open problem in linear optimization whether there is a version of the simplex method that runs in polynomial time, and if the polynomial Hirsch conjecture is false, the answer to that question is unconditionally no. 

In an effort to make progress on this very difficult problem, several variants and abstractions of the polynomial Hirsch conjecture have appeared in the literature. Open variants include the polynomial monotone Hirsch conjecture motivated by the simplex method's requirement that the objective function increases at each step, the strict monotone Hirsch conjecture by Ziegler in \cite{LecturesonPolytopes}, and the abstract Hirsch conjecture by Kalai for the polymath3 project motivated by the connected layer families introduced in \cite{LimitsofAbstraction}. A positive resolution to any of those would imply the polynomial Hirsch conjecture. Other variants have been resolved including the circuit diameter conjecture of \cite{CircuitDiameterConjecture} recently proven in its polynomial version in \cite{StronglyPolyCircuitDiam} and the continuous variant for interior point methods from \cite{ContinuousDStep} that turned out to be false by breakthrough work in \cite{IPMsNotStronglyPoly}. The analogous question for simplicial complexes is known to have a negative answer even for pseudomanifolds \cite{psueodomanifoldhirsch, simpcompdiameter} and a positive resolution for flag normal simplicial complexes \cite{FlagNormComplex}. It is open for simplicial manifolds with a notable case being that of simplicial spheres. For the few abstractions in which a resolution is known, the paths involved differ fundamentally from those in graphs of polytopes. For example, in the simplicial complex and pseudomanifold case, the resulting graph is a path. For circuit diameters and the continuous variant, the path traverses the interior of the polyhedron.

One approach to study these problems is to study the space of all such paths. In particular, to prove the polynomial Hirsch conjecture, it would suffice to prove its strict monotonic variant (See the discussion surrounding Conjecture $3.9$ in Chapter $3$ of \cite{LecturesonPolytopes}), which asks whether for any linear function $\mathbf{c}^{\intercal} \mathbf{x}$, there is a path of polynomial length from some minimizer of $\mathbf{c}^{\intercal}\mathbf{x}$ to some maximizer of $\mathbf{c}^{\intercal} \mathbf{x}$. We will refer to such paths as \textbf{monotone paths}. There is a topological space one can associate to the set of monotone paths called the Baues complex due to its appearance in the case of hyper-cubes in work of Baues in algebraic topology \cite{BauesOriginal}. In this context, the so-called Baues problem asked whether this space was a topological sphere. Billera, Kapranov, and Sturmfels showed the answer was yes in \cite{CellularStrings} by showing that the Baues complex deformation retracts onto the boundary complex of a polytope called the monotone path polytope, a key special case of the fiber polytopes construction by Billera and Sturmfels in \cite{FiberPolytopes}. The vertices of the monotone path polytope correspond to a subset of monotone paths called \textbf{coherent monotone paths}, and as a consequence of these results, these coherent monotone paths completely capture the topology of the space of all monotone paths for a fixed choice of $\mathbf{c}$. 

To be more explicit, the coherent monotone paths are precisely those that arise from a parametric linear optimization problem, which relates to their appearance in toric geometry where they correspond to coset representatives of quotients of toric varieties by one dimensional subtori \cite{ToricQuotients}. We have a polytope $P = \{\mathbf{x} : A \mathbf{x} \leq \mathbf{b}\}$ and a linear objective $\mathbf{c}$ that gives the orientation of the graph. A monotone path walks from a minimum to a maximum of that objective. Given a generic objective $\mathbf{w}$, the set of optimal faces for across all choices of $\lambda \in \mathbb{R}$ of
\[\max((\mathbf{w} + \lambda \mathbf{c})^{\intercal} \mathbf{x}) \text{ s.t. } A \mathbf{x} \leq \mathbf{b}\]
are all either vertices or edges and corresponds to a monotone path. A monotone path is called \textbf{coherent} if it arises in this way. In the geometric combinatorics literature, coherent monotone paths have chiefly been studied on nice families of polytopes such as simplices and cubes \cite{FiberPolytopes}, zonotopes \cite{edmanthesis, ZonotopeCoherence}, piles of cubes \cite{PilesofCubes}, hyper-simplices \cite{Hypersimps, MPPHyperSimp}, cyclic polytopes \cite{CyclicFiber}, matroid and polymatroid polytopes \cite{FlagPolymatroids}, and random polytopes \cite{CoherentUnimodality}. 

Given this context, we ask the following question: Is the polynomial Hirsch conjecture true for coherent monotone paths? That is, does there always exist a coherent monotone path of polynomial length for any given polytope and linear function $\mathbf{c}^{\intercal} \mathbf{x}$. It turns out the answer is no. Even the Kalai-Kleitman bound does not hold in the coherent setting. 
\begin{theorem}
\label{thm:main}
For each $d \geq 3$, there exists a $d+1$ dimensional polytope $P$ with $(2d+1)(d+1)+1$ facets and a linear objective function such that any coherent monotone path on $P$ with respect to that objective function is of length at least $2^{d-1}-1$.
\end{theorem}

It already follows from 1973 work of Zadeh on network flow problems \cite{ZadehBadFlow} that a coherent monotone path may be exponentially long. This result appeared also in work by Murty in parametric linear optimization \cite{Mur80}, Goldfarb in the context of the shadow simplex method \cite{report/Goldfarb83}, and Amenta and Ziegler in discrete and computational geometry \cite{jour/cm/AZ98}. However, a single coherent monotone path being exponentially long is only weak evidence toward our result as polytopes generally have exponentially many coherent monotone paths. 

In the context of the shadow simplex method in \cite{black2024exponentiallowerboundspivot}, they studied a more general notion of coherent path that allowed an arbitrary start point instead of the minimum. There they proved exponential lower bounds by arguing that a certain choice of starting point can be made so constrained that there is essentially only one path that one may choose. When starting from the minimum and going to the maximum, such a trick no longer works. We require fundamentally new ideas to overcome this barrier. Most importantly, the topological results by Billera, Kapranov, and Sturmfels \cite{CellularStrings} do not apply in the same way to the notion of coherent paths they consider. They do not capture the topology of all such paths. To be completely explicit, to our knowledge, ours are the first and only lower bounds for the notion of coherent monotone paths introduced by Billera and Sturmfels in \cite{FiberPolytopes}. 

Our result is surprising for a few reasons. First, there is broad consensus due in part to an enormous amount of computational evidence from the simplex method literature that the polynomial Hirsch conjecture should be true. Our result lies in sharp contrast to all of that evidence.

Second there are polytopes such as the hyper-cube \cite{CellularStrings} for which every monotone path is coherent. More generally, the space of coherent monotone paths is numerous. The monotone path polytope is, up to normal equivalence, the Minkowski sum of all sections of the polytope orthogonal to the objective that contain a vertex (See Definition A.2.1 of \cite{AlexThesis}) meaning that it is generically a Minkowski sum of exponentially many polytopes and thus tends to have many vertices. Since those vertices are in bijection with coherent monotone paths, polytopes tend to have many coherent monotone paths. Our lower bound construction does not make any effort to limit the breadth of the set of coherent monotone paths and instead relies on ensuring that any of the numerous such paths must be long. This is in contrast to lower bounds for pivot rules for the simplex method, where they generally only need to show that one path is long. 

Third any short monotone path in our example must not only not be coherent but be far from being coherent in a measurable sense. In the Baues complex, if two monotone paths are adjacent, they agree outside of a two-dimensional face of the polytope. This graph is called the \textbf{flip graph} of monotone paths, and its connectivity was studied in \cite{MontonePathFlipGraph, FlipGraphLattices, MonotonePathEnumeration,OrientedMatroidsMonoPaths, CoxeterGroupMonoPaths, edmanthesis}. As a consequence of the definition, the difference in number of edges between a pair of adjacent monotone paths is at most the maximum number of edges of a two dimensional face of the polytope, which is at most the number of facets of the polytope. Thus, as a consequence, we have the following corollary:
\begin{cor}
For each $d \geq 3$, there exists a $d+1$ dimensional polytope $P$ with $(2d+1)(d+1)+1$ facets and a linear objective function such that any monotone path of polynomial length is superpolynomially many flips away from any coherent monotone path.
\end{cor}

\subsection{Algorithmic and Geometric Consequences}

The fourth reason our result is surprising relates to where coherent monotone paths appear in the context of the simplex method. They are related to the paths traced by the shadow simplex method introduced in the 1950's by Gass and Saaty \cite{gas55}. Essentially all of the strongest positive results for finding short paths on polytopes use the shadow simplex method such as the average case analysis of Borgwardt \cite{thesis/Borgwardt77}, the smoothed analysis of Spielman and Teng \cite{ST04} recently optimized in \cite{optimal_smoothed_analysis}, the best known polynomial bound for lengths of paths in polytopes with well conditioned constraint matrices in \cite{dadushhahnle, conf/icalp/BR13}, and the by the book analysis in \cite{ByTheBook}. In fact, in \cite{ByTheBook, dadushhahnle, optimal_smoothed_analysis}, they use what they call a \textbf{semi-random shadow}. There they start with a fixed objective $\mathbf{w}$ and randomly sample an auxiliary objective $\mathbf{c}$ and trace the path of optima of $\mathbf{w} + \lambda \mathbf{c}$ such that $\lambda > 0$. Then they bound the expected length of such a path. 

By flattening a polytope, one can use the existence of a single exponentially long coherent path to show that the expected length a random shadow of a polytope can be exponential. However, flattening the polytope can be modeled purely as changing the probability distribution, and thus this average case lower bound depends strongly on the distribution. Here we prove a robust lower bound for semi-random shadows independent of the distribution as expressed in the following corollary.

\begin{cor}
\label{cor: shadowsimplex}
For each $d \geq 3$, there exists a $d+1$ dimensional polytope $P$ with $(2d+1)(d+1)+1$ facets and a linear objective function $\mathbf{w}$ such that for any choice of $\mathbf{c}$ orthogonal to $\mathbf{w}$ the shadow simplex path from $\mathbf{w}$ to $\mathbf{c}$ is of length at least $2^{d-1}-1$.
\end{cor}

One can interpret this picture dually. Namely, recall that each face of a polytope has a \textbf{normal cone} consisting of all objectives optimal at that face with rays given by the rows of inequalities tight at that face. The \textbf{normal fan} of the polytope is the union of all such cones. Then the parametric optimization problem for maximizing in the direction $\mathbf{w} + \lambda \mathbf{c}$ has a given vertex in its set of solutions if and only if that line intersects the normal cone of that vertex. Taking the normal fan and intersecting it with the affine hyperplane orthogonal to $\mathbf{w}$ and containing $\mathbf{w}$ induces a regular triangulation of a point configuration. As a consequence of our construction method, we have the following result for stabbing convex subdivisions:

\begin{cor}
\label{cor:stabbing}
For each $d \geq 3$, there exists a regular triangulation of a point configuration with $(2d+1)(d+1)$ points and a point in the interior of one of the simplices of the triangulation such that every ray emanating from that point intersects at least $2^{d-1}$ simplices of the triangulation. 
\end{cor}

Note this is not immediate from Theorem \ref{thm:main} but follows directly from our construction method. Stabbing convex subdivisions has a long history in discrete and computational geometry (see \cite{RayShooting, GeometricTransversals}). Extending work of Amenta and Ziegler \cite{jour/cm/AZ98}, Shewchuk showed in \cite{StabDelaunay} that a line can intersect a number of simplices in a Delaunay triangulation of a point configuration exponential in the number of vertices. Our result is a substantial strengthening of theirs in the more general context of regular triangulations.

A related consequence occurs for parametric linear optimization, the main algorithmic context in which coherent paths secretly appear. In parametric linear optimization with one parameter, one wants to find the set of all optima for objectives of the form $\mathbf{w} + \lambda \mathbf{c}$ such that $\lambda \in \mathbb{R}$ for $\mathbf{w}, \mathbf{c} \in \mathbb{R}^{d}$ fixed. Even for shortest path problems \cite{carstensen, MulmuleyShah, GajjarRadhakrishnan} and more generally network flow problems \cite{NPmighty, ZadehBadFlow}, the number of parametric optima may be superpolynomial. However, these individual examples require highly pathological choices of $\mathbf{c}$. In particular, in the smoothed analysis of the successive shortest path algorithm, they note that their analysis shows that perturbing the auxiliary objective is enough to obtain polynomial bounds for network flow problems \cite{SmoothedShortPaths}. Kelner and Nikolova observed a similar phenomenon for parametric optimization over lattice polytopes in Lemma 2.10 of \cite{kelnernikolova} in the context of smoothed analysis of low rank convex optimization. The question we ask here is whether such a smoothed analysis yields a polynomial bound for general parametric linear programs, and the following corollary says the answer is no independent of the probability distribution for perturbation:

\begin{cor}\label{cor:parametric optimization}
For each $d \geq 3$, there exists a $d+1$ dimensional polytope $P$ with $(2d+1)(d+1)+1$ facets and a linear objective function $\mathbf{w}$ such that for any choice of $\mathbf{c}$ linearly independent from $\mathbf{w}$ the parametric optimization problem $\mathbf{w} + \lambda \mathbf{c}$ has at least $2^{d-1}$ many solutions.
\end{cor}

Thus, this shows parametric linear optimization is very hard, and the origin of its difficulty is not purely in the choice of auxiliary objective but can come entirely from the underlying linear program. This gives further justification for the use of approximation algorithms for parametric optimization problems as surveyed in detail in \cite{ParametricOptSurvey}. 

To prove these results, we give a novel construction method. The idea is to start with a single polytope with one exponentially large shadow and apply a transformation to make it so for a fixed $\mathbf{w}$ there is a large cone such that for all $\mathbf{c}$ in that cone, the set of maximizers of $\mathbf{w} + \lambda \mathbf{c}$ with $\lambda > 0$ is exponential. The cone is large enough so that $d+1$ rotated copies of it cover the entire space, so we take $d+1$ rotated copies of the polytope and glue them together. Then the resulting polytope satisfies that every ray of the form $\mathbf{w} + \lambda \mathbf{c}$ with $\mathbf{c}$ orthogonal to $\mathbf{w}$ intersects exponentially many cones, which implies all of our bounds. The core nontrivial step is gluing everything together. All technicalities of the proof are to accommodate a core lemma (Lemma \ref{lem:refinement}) giving a sufficient condition allowing one to find a common refinement of a pair of normal fans while only increasing the number of facets additively in contrast to the large growth one can see when using Minkowski sums. See Figure \ref{fig:ProofIdea} for a visual description of the proof idea and Figure \ref{fig:refinement} for an illustration of Lemma \ref{lem:refinement}.

\section{Basic tools and operations from convex geometry}

We first recall several foundational definitions and geometric operations from the theory of polyhedra.

A \textbf{polyhedron} $P \subseteq \mathbb{R}^d$ is the intersection of finitely many closed halfspaces, defined by a system of linear inequalities:
\[P = \{\x \in \R^d : A\mathbf{x} \leq \mathbf{b}\}\]
where $A \in \mathbb{R}^{n \times d}$ and $\mathbf{b} \in \mathbb{R}^n$. A bounded polyhedron is called a \textbf{polytope}. Equivalently, a polytope is the convex hull of a finite set of points in $\R^d$. An inequality $\mathbf{A}_i \mathbf{x} \leq b_i$ from the system is said to be \textbf{irredundant} if its removal changes the polyhedron $P$. 

If a linear inequality $\mathbf{a}^\T \x \leq b$ holds for all points $\x \in P$ and, in addition, holds as an \emph{equality} for at least one point in $P$, then the hyperplane 
\[H= \{ \x\in\R^d:\; \mathbf{a}^\T\x = b\} \]
is called a \textbf{supporting hyperplane} of $P$. Geometrically, $H$ touches $P$ on the boundary without cutting into its interior. A subset of points $F\subseteq P$ is a \textbf{face} of $P$ if it is the intersection between $P$ and a supporting hyperplane. Every face of $P$ is itself a lower-dimensional polyhedron. The dimension of a face, $\dim(F)$, is the dimension of its affine hull. A \textbf{facet} of $P$ is a maximal proper face of $P$, i.e., its dimension is exactly one less than the polyhedron itself. 

If $P$ is full-dimensional in $\R^d$ and the inequalities in $A\x\leq \mathbf{b}$ are irredundant, then there is a one-to-one correspondence between the inequalities in the system and the facets of $P$: each row inequality $\mathbf{A}_i \mathbf{x} \leq b_i$ uniquely defines a supporting hyperplane 
\[H_i = \{ \x\in\R^d\;:\; \mathbf{A}_i\x = b_i \}, \]
and the intersection $F_i = P \cap H_i$ is a facet. The row vector $\mathbf{A}_i$ is an \textbf{outer normal} of this facet $F_i$ (pointing strictly outward from the interior of $P$).

Throughout this paper, we use the notation that a vector $\x \in \R^d$ is pointing \textbf{upward} if its last coordinate is positive, i.e., $\x^{\T} \e_d>0$. A normal cone is called \textbf{upward facing} if all its generating rays are. The \textbf{upper hull} of a polyhedron $P$ is the union of the facets with upward-pointing outer normals.

For any $\u \in \mathbb{R}^d$, the \textbf{support function} of $P$ is defined as $h_P(\u) := \max_{\x \in P} \u^\T \x$. A point $\mathbf{v} \in P$ is called the \textbf{maximizer of a vector $\u$} if the maximum of the support function is achieved at $\v$, i.e., $h_P(\u) = \u^\T \v$.

 The geometric intuition for our construction comes from the \textbf{normal fan} of a polyhedron. 
 \begin{definition}
   Let $P \subset \R^d$ be a polyhedron. The \textbf{normal cone} of a vertex $\v$ of $P$, denoted $N_\v$, is the set of all vectors $\u \in \R^d$ such that the linear functional $\x \mapsto \u^{\T} \x$ achieves its maximum over $P$ at the point $\v$. That is,
   \[N_\v := \{ \u \in \mathbb{R}^d : \u^{\T} \v \ge \u^{\T} \x \text{ for all } \x \in P \}.\]

   Equivalently, $N_\v$ is the conic hull generated by the outer normals of all facet defining inequalities that are tight at $\v$. 
   The \textbf{normal fan} of a polyhedron $P$, denoted $\mathcal{N}(P)$, is defined here as the collection of the normal cones of its vertices. The normal fan $\mathcal{N}(Q)$ of another polyhedron $Q$ is said to \textbf{refine} the normal fan of $P$ if each normal cone of $P$ is a union of normal cones of $Q$.  
   
   The standard definition of a normal fan as in Example $7.3$ of Chapter $7$ of \cite{LecturesonPolytopes} is actually defined to include the normal cones of all faces, not just the vertices, and the normal cones of vertices are exactly the maximal, full-dimensional cones in the fan. We adopt this non-standard definition that ignores higher-dimensional faces, since we will only be focusing on unique vertex maximizers of linear functions and thus only care about the full-dimensional normal cones. For brevity, throughout the paper, any reference to normal cones refers to the full-dimensional ones corresponding to vertices.
 \end{definition}

The following proposition captures the correspondence between maximizers of linear functions and normal cones of polyhedra, linking the language of linear programming directly to polyhedral geometry.

 \begin{prop}\label{prop: maximizer vs normal cone}\cite[Section 7.1]{LecturesonPolytopes}
   Let $P = \{\mathbf{x} \in \mathbb{R}^d : A\mathbf{x} \leq \mathbf{b}\}$ be a full-dimensional polyhedron, $\mathbf{v}$ be a vertex of $P$, and let $I_{\mathbf{v}} = \{i : \mathbf{A}_i \mathbf{v} = b_i\}$ denote the index set of facet-defining inequalities that are tight at $\mathbf{v}$. For any objective vector $\mathbf{u} \in \mathbb{R}^d$, the following statements are equivalent:
   
   \begin{enumerate}
   \item $\mathbf{u}$ is in the normal cone of $\mathbf{v}$, i.e., $\mathbf{u} \in N_\mathbf{v}$.
   
   \item $\mathbf{v}$ is a maximizer of the linear function $\mathbf{x} \mapsto \mathbf{u}^\T \mathbf{x}$ over $P$.
   
   \item $\mathbf{u}$ is in the cone of the outer normals that are tight at $\mathbf{v}$. In other words, there exist non-negative scalars $\lambda_i \geq 0$ for $i \in I_{\mathbf{v}}$ such that:$$\mathbf{u} = \sum_{i \in I_{\mathbf{v}}} \lambda_i \textbf{A}_i$$
   \end{enumerate}
 \end{prop}

 We use these connections in order to understand the geometry of paths. Let $P\subset \R^d$ be a polyhedron and $\mathbf{c} \in \mathbb{R}^{d}$. A \textbf{coherent cellular string} \cite{CellularStrings} with respect to $\mathbf{c}$ is the set of faces of $P$ that maximize the parametric linear objective
 \[ \max_{\x\in P}(\w+\lambda\c)^\T \x\]
 for some $\mathbf{w}$ linearly independent of $\mathbf{c}$ across all choices of $\lambda \in \R$. If the set of faces are always either vertices or edges, the coherent cellular string is called a \textbf{coherent monotone path}. This is the case for any generic choice of $\mathbf{w}$.  

 The following proposition shows that tracking a coherent monotone path is equivalent to tracking the sequence of normal cones crossed through by a ray. This connection makes the normal fan the natural geometric setting for our arguments. 

 \begin{prop}
 \label{prop: coherentpath}
   Let $P\subset \R^d$ be a polyhedron, and let $\c, \w$ be two linearly independent vectors in $\R^d$. Define the line
   \[L = \{\w + \lambda\c : \lambda  \in \mathbb{R}\}. \]
   Then, if the set of optimizers of $L$ form a coherent monotone path, a vertex $\v$ of $P$ is in the coherent monotone path induced by $L$ if and only if the line crosses through its normal cone $N_\v$. 
 \end{prop}

 We can simplify our choice of the initial point $\w$ by assuming it is orthogonal to $\c$ without loss of generality. 

 \begin{prop}\label{prop: orthogonality}
   Let $P\subset \R^d$ be a polyhedron, and let $\c, \w$ be two linearly independent vectors in $\R^d$. Define the line
   \[L= \{\w + \lambda\c : \lambda\in \R \}.\]
   Then there exists a new vector $\w'$ such that $\w'$ is orthogonal to $\c$, and the line $\{\w' + \lambda\c : \lambda\in \R \}$ is identical to $L$.
 \end{prop}

Since rescaling does not change the set of normal cones intersected, the set of normal cones intersected by $\{\w' + \lambda \c: \lambda > 0\}$ is the same as that of $\{\lambda \w' + \c: \lambda > 0\}$.

For our construction, we need to understand how linear transformations act on a normal fan. The following lemma acts as our fundamental mechanical tool. This is a folklore lemma, but we include a proof here for self containment. 
\begin{lem}
\label{lem:normalfantransform}
Let $T: \mathbb{R}^{d} \to \mathbb{R}^{d}$ be an invertible linear transformation and $P$ be a polytope. Applying $T$ to the normal fan $\mathcal{N}(P)$ is equivalent to applying the inverse transpose to the polytope, i.e.,
\[ T(\mathcal{N}(P)) = N\big(((T^{-1})^{\T}(P)\big).  \]

\end{lem}

\begin{proof}

For any face $F$ of $P$, let $N_F$ be its normal cone, and let $P = \{\mathbf{x}: A\mathbf{x} \leq \mathbf{b}\} \subseteq \mathbb{R}^{d}$ be an irredundant inequality description of the polytope. Let $\mathbf{A}_{1}, \mathbf{A}_{2}, \dots, \mathbf{A}_{n}$ be facet normals of the polytope. Observe that:
\begin{align*}
  (T^{-1})^{\T} P &= \{(T^{-1})^{\T} \mathbf{x} \in \mathbb{R}^{d}: A\mathbf{x} \leq \mathbf{b} \} \\
  &= \{\mathbf{y}\in \mathbb{R}^{d}: A ((T^{-1})^{\T})^{-1} \mathbf{y} \leq \mathbf{b}\} \\
  &= \{\mathbf{y} \in \mathbb{R}^{d}: A T^{\T} \mathbf{y} \leq \mathbf{b}\} \\
  &= \{\mathbf{y} \in \mathbb{R}^{d}: \begin{pmatrix} (TA_{1})^{\T} \\(TA_{2})^{\T} \\\vdots\\ (TA_{n})^{\T} \end{pmatrix} \mathbf{y} \leq \mathbf{b}\}.
\end{align*}
Note that $(T^{-1})^{\T}$ is a linear isomorphism and so preserves the combinatorics of the polytope mapping the facet normal $A_{i}$ to $TA_{i}$. Via this map, we obtain that the resulting normal fan is precisely $\{T C: C \in \mathcal{N}\}$.
\end{proof}

The following corollary shows that we can use a map that flattens the polytope and conversely squeezes the normal fan of a polyhedron to the neighborhood of a line without changing the combinatorial structure.
\begin{cor}\label{cor: squeezing normal fan}
 Let $P\subset \R^d$ be a polyhedron, $\v$ be a unit vector in $\R^d$, and $\epsilon >0$. There exists a polyhedron $Q$ combinatorial equivalent to $P$ such that every unit generating vector $\u$ in $\mathcal{N}(Q)$ is arbitrarily close to $\v$, i.e., $\min\{ |\u + \v|, |\u-\v|\}<\epsilon$.
\end{cor}

\begin{proof}
Up to a rotation, it suffices to prove the statement for the unit vertical vector $\v = \e_d$. Let $T$ be the linear transformation given by the diagonal matrix $\text{diag}(1, \dots, 1, c)$ for some $0 < c < 1$. This map compresses the last coordinate while preserving its orthogonal complement. Let $Q = T(P)$. By Lemma \ref{lem:normalfantransform}, its normal fan is 
\[\mathcal{N}(Q) = \text{diag}(1, \dots, 1, 1/c) \mathcal{N}(P),\] so choosing $c$ sufficiently small will make $1/c$ arbitrarily large. After normalizing, the last coordinate of each generating ray dominates all other horizontal coordinates. which squeezes the 
(normalized) generating rays to fall within an $\epsilon$-neighborhood of $\e_{d+1}$.
\end{proof}

\section{Constructing Hard Instances}

\subsection{Planting the Seed}

To complete our construction, we start with an initial building block which we call a \textbf{seed polyhedron}. We will later rotate, compress, translate, and intersect it with copies of itself to build our pathological instance. The following definition formalizes the requirements we need for the construction. 

\begin{definition}\label{def: seed}
  A $d+1$-dimensional polyhedron $\tilde{Q}$ is called a \textbf{seed polyhedron} if all of the following conditions hold.
\begin{enumerate}
  \item $\tilde{Q}$ has a number of facets polynomial in $d$. 
  \item The outer normals of its facets all point upward. 
  \item There exists $\mathbf{c} \in \mathbb{R}^{d}$ such that $\{(\lambda \mathbf{c}, 1): \lambda \geq 0\}$ intersects exponentially many normal cones of $\tilde{Q}$.
  \item There is a unique vertex that maximizes $\e_{d+1}$.

\end{enumerate}
\end{definition}

The existence of polytopes that have a $2$-dimensional projection with many vertices is a classical result shown by many others such as Murty \cite{Mur80}, Goldfarb \cite{report/Goldfarb83}, and Amenta and Ziegler \cite{jour/cm/AZ98}:

\begin{lem}
There exists a $d$-dimensional polytope with $2d$ facets and a projection map $\pi: \R^d \to \R^2$ such that the polygon $\pi(P)$ has $2^{d}$ vertices.   
\end{lem}

Via the standard equivalence between projection and parametric optimization as is expressed in Lemma 7.11 of \cite{LecturesonPolytopes}, the set of vertices of $\pi(P)$, where $\pi(\mathbf{x}) = (\mathbf{c}^{\intercal} \mathbf{x}, \mathbf{w}^{\intercal} \mathbf{x})$, are the set of maximizers of $\lambda_{1} \mathbf{c} + \lambda_{2} \mathbf{w}$ such that $\lambda_{1}, \lambda_{2} \in \mathbb{R}$ are generic and in particular not both zero.

The upper hull of the polygon consists of those vertices with facet normal with positive second coordinate. Since the second coordinate is given by $\mathbf{w}^{\intercal} \mathbf{x}$ under the projection $\pi$, these are precisely the maximizers of $\lambda_{1}\mathbf{c} + \lambda_{2}\mathbf{w}$, where $\lambda_{2} > 0$. Up to the choice of second coordinate of the projection as $\mathbf{w}$ or $-\mathbf{w}$, one can without loss of generality find a projection such that half the vertices are on the upper hull. Furthermore, by possibly replacing $\mathbf{w}$ with $\mathbf{w} - \lambda \mathbf{c}$ for some sufficiently large $\lambda > 0$, we can also assume that each vertex on the upper hull is a unique maximizer of $\mathbf{w} + \lambda \mathbf{c}$ for some $\lambda \geq 0$. This implies the following corollary.

\begin{cor}
\label{cor:upperhull}
There exists a $d$-dimensional polytope $P \subseteq \mathbb{R}^{d}$ with $2d$ facets and $\mathbf{c}, \mathbf{w} \in \mathbb{R}^{d}$ such that the ray $\{\mathbf{w} + \lambda \mathbf{c}: \lambda \geq 0\}$ intersects the interior of at least $2^{d-1}$ normal cones of $P$.

\end{cor}

\begin{construction}[Construction of the seed polyhedron $\tilde{Q}$]\label{construction: Q tilde}

Let $P = \{\mathbf{x}: A \mathbf{x} \leq \mathbf{b}\} \subseteq \mathbb{R}^{d}$ be a $d$-dimensional polytope and $\mathbf{c}, \mathbf{w}$ be in $\mathbb{R}^{d}$ such that $\{\mathbf{w} + \lambda \mathbf{c}: \lambda \geq 0\}$ intersects the interior of at least $2^{d-1}$ normal cones of $P$ as in Corollary \ref{cor:upperhull}. 

First, lift $P$ to one dimension higher by letting 
\[
\tilde{P} := \{(\mathbf{x}, \mathbf{w}^{\intercal} \mathbf{x}) : \mathbf{x} \in P\} \subseteq \mathbb{R}^{d+1}.
\] 
Geometrically, $\tilde{P}$ completely lies inside the hyperplane $H := \{\y=(\mathbf{x}, z) \in \mathbb{R}^{d+1} : \mathbf{w}^\intercal \mathbf{x} - z = 0\}$, where $z$ denotes the $(d+1)$-th coordinate. 

Let $\x_0$ be a point in the interior of $P$, so $A\x_0< \mathbf{b}$. For each vertex $\mathbf{v} \in P$ that is projected to the upper hull of $\pi(P)$, choose a scalar $\mathbf{\lambda}_{\mathbf{v}} \in \mathbb{R}_{>0}$ such that $\v$ is the unique maximizer of the linear function $(\mathbf{w}+\lambda_{\v} \mathbf{c})^\T \x$ over $P$. Then we define a sufficiently large bounding height $M$ by 
\[M := 2\max(\{\mathbf{w}^{\intercal} \mathbf{x}: \mathbf{x} \in P\} \cup \{\lambda_{\mathbf{v}}\mathbf{c}^{\intercal}(\mathbf{v}-\mathbf{x}_0): \mathbf{v} \in U(P)\},0),\]
where $U(P)$ is the set of vertices projected to the upper hull of $\pi(P)$. Since $\mathbf{w}^{\intercal}\x_0 < M+1$, the point $\y^* = (\x_0, M+1)$ lies strictly above $H$ in the last coordinate. Therefore we can cone over $\tilde{P}$ with apex $\y^*$ to get a pyramid $Q$:
\[Q := \text{conv}(\tilde{P}\cup \y^*).
\] 
By construction, $\tilde{P} = Q \cap \{ \y=(\mathbf{x},z) \in \mathbb{R}^{d+1}: z = \mathbf{w}^{\intercal} \mathbf{x}\}$, and so there exists a $2d \times (d+1)$ matrix $A'$ and a vector $\mathbf{b}' \in \mathbb{R}^{2d}$ such that the pyramid $Q$ can be written (without redundancy) as:
\[
Q = \{\y = (\mathbf{x},z)\; :\; A'\y \leq \mathbf{b}', \mathbf{w}^{\intercal} \mathbf{x} \leq z\}.
\]
Finally, we define our candidate of the seed polyhedron for our construction by flipping the halfspace related to $H$ and let
\begin{align}\label{def: Q tilde}
\tilde{Q} :=\{\y=(\mathbf{x},z) \; :\; A'\y \leq \mathbf{b}', \mathbf{w}^{\intercal} \mathbf{x} \geq z\}.
\end{align}
Geometrically, this corresponds to taking the pointed (unbounded) cone with apex vertex $\mathbf{y}^{\ast}$ generated by the rays passing through the lifted polytope $\tilde{P}$, and intersecting it with the halfspace not containing $\mathbf{y}^{\ast}$ with bounding hyperplane through $\tilde{P}$. See Figure \ref{fig:theSeed} for an illustration.
\end{construction}

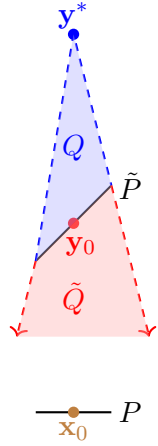
\begin{figure}
  \centering
  
\[\begin{tikzpicture}


  \draw[thick, black] (2,0) -- (3,1); 

  \draw[thick, black] (2,-2) -- (3,-2);
  
  \draw[thick, blue, dashed] (2,0) -- (2.5,3)-- (3,1);

  \draw[thick, red, ->, dashed] (2,0) -- (1.75, -1);
  
  \draw[thick, red, ->, dashed] (3,1) -- (3.5, -1);

  \draw (2.5,.5) node[red, circle, fill, inner sep = 1.5pt] {};

  \draw (2.5,-2) node[brown, circle, fill, inner sep = 1.5pt] {};

  \draw (2.5,3) node[blue, circle, fill, inner sep = 1.5pt] {};
  
  \fill[blue!30, opacity=0.4] (2,0) -- (3,1) -- (2.5,3) --cycle; 
  \fill[red!30, opacity=0.4] (1.75,-1) -- (2,0) -- (3,1) -- (3.5,-1) --cycle; 

  \draw (2.5,-2.25) node[brown] {$\mathbf{x}_{0}$};

  \draw (2.6,.2) node[red] {$\mathbf{y}_{0}$};
  
  \draw (2.5,3.25) node[blue] {$\mathbf{y}^{\ast}$};

  \draw (2.5,1.5) node[blue] {$Q$};

  \draw (2.5,-.5) node[red] {$\tilde{Q}$};

  \draw (3.25,-2) node[black] {$P$};

  \draw (3.25,1) node[black] {$\tilde{P}$};
\end{tikzpicture}\]
  \caption{Pictured is a schematic diagram of the seed construction. We start with $P$ and a strictly feasible point $\mathbf{x}_{0}$ and lift it one dimension up with height function $\mathbf{w}^{\intercal}\mathbf{x}$ to obtain $\tilde{P}$ and the strictly feasible $\mathbf{y}_{0} \in \tilde{P}$. Then we lift $\mathbf{y}_{0}$ sufficiently higher to obtain $\mathbf{y}^{\ast}$, and $Q = \text{conv}(P \cup \mathbf{y}^{\ast})$. Finally to obtain $\tilde{Q}$, we flip the inequality defining $\tilde{P}$ as a facet of $Q$. The result is an unbounded polyhedron. }
  \label{fig:theSeed}
\end{figure}

The remaining lemmas in this section are devoted to prove that $\tilde{Q}$ satisfies all conditions in Definition \ref{def: seed}.

\begin{lem}
\label{lem:allpositive}
The outer normals of facets of $\tilde{Q}$ all have a positive last coordinate, i.e., they all point upward.
\end{lem}

\begin{proof}
By (\ref{def: Q tilde}), the facets of $\tilde{Q}$ fall into two categories. One facet ($\tilde{P}$) is restricted by the inequality $\mathbf{w}^{\intercal} \mathbf{x} \geq z$, so it has the strictly upward-pointing outer normal $(-\mathbf{w}, 1)$. Each remaining facet is given by a row of the inequality $A'\y \leq \mathbf{b}'$. At row $i$, an outer normal vector of the inequality $A'\y \leq \mathbf{b}'$ is just $\n_i$, $i$-th row of $A'$. By construction, the apex $\mathbf{y}^{\ast} = (\x_0, M+1)$ lies on the supporting hyperplanes of all of these facets, and there is a point $\y_0 = (\x_0, \w^\T\x_0) \in \text{int}(\tilde{P})$ directly underneath it.
Hence there is a vertical downward-pointing vector from $\y^*$ down to $\y_0$ given by $\alpha \e_{d+1}$ with $\alpha < 0$. Because this vector points from a boundary point ($\y^*$) into the interior of the side constraints ($\y_0$), its inner product with each outer normal $\n_i$ must be strictly negative, so
\[ (\y_0 - \mathbf{y}^{\ast})^\T \mathbf{n}_i = \alpha \e_{d+1}^\T \n_i< 0 \]
Since $\alpha < 0$, it follows that the final coordinate of the normal vector must be positive for each $i$, as desired.
\end{proof}

Since all the outer normals for $\tilde{Q}$ point upwards, starting from any point $\y\in\tilde{Q}$ and shooting a ray down in the $-\e_{d+1}$ direction always yields negative inner product with every outer normal vector, therefore the ray never escapes $\tilde{Q}$. This implies:

\begin{cor}
The polyhedron $\tilde{Q}$ is unbounded.
\end{cor}

In the rest of this paper, we will often embed $\x\in \R^d$ into the hyperplane at height $1$ in $\R^{d+1}$. Let's introduce some clean notation for this lifting.

\begin{notation}
For $\x\in \R^d$, we denote the lift of $\x$ into affine slice $\{ y_{d+1} =1\} \subset \R^{d+1}$ by
   \[\bar{\x} := (\x, 1). \]
   Similarly, we denote the horizontal ray starting at the point $\e_{d+1}$ and extending in the direction of $\x$ within the slice by
   \[ \r(\x) := \{(\lambda \x, 1): \lambda\geq 0 \}. \]
\end{notation}

We now try to understand the $2$-dimensional projection $\pi(\x) = (\c^\T\x, \w^\T\x)$. Recall from Corollary \ref{cor:upperhull} that every vertex $\v\in P$ is the unique maximizer of linear function $(\w + \lambda_\v \c )^\T\x$ over $P$ for some $\lambda_\v$.

\begin{lem}
\label{lem:liftupperhull}
Let $\mathbf{v}$ be a vertex of $P$, and suppose $\mathbf{v}$ is in the upper hull of $\pi(P)$. Then the lifted vertex $\y_\v=(\mathbf{v}, \mathbf{w}^{\intercal} \mathbf{v})$ is the unique maximizer of linear function $\overline{\lambda_{\mathbf{v}} \mathbf{c}}^\T\y$ over $\tilde{Q}$.
\end{lem}

\begin{proof}
We will show that moving along any edge/ray direction on $\tilde{Q}$ emanating from $\y_\v$ strictly decreases the value of $\overline{\lambda_{\mathbf{v}} \mathbf{c}}^\T\y$. By construction, $\tilde{Q}$ is the intersection of the cone over $\tilde{P}$ with a halfspace and $\y_\v$ is a vertex of $\tilde{P}$. Therefore the directions going out of $\y_\v$ fall into the following two categories: 

First, the neighboring vertices of $\y_\v$ in $\tilde{Q}$ are precisely its neighbors in $\tilde{P}$. Since $\v$ uniquely maximizes $(\w+ \lambda_\v\c )^\T \x$ over $P$, we have $\y_\v$ uniquely maximized $\overline{\lambda_{\mathbf{v}} \mathbf{c}}^\T\y$ in $\tilde{P}$. Therefore all the edges from $\y_\v$ in $\tilde{P}$ strictly decreases the objective function.

Second, the only one ray on the boundary of $\tilde{Q}$ that starts at $\y_\v$ is the ray that previously come from the apex $\mathbf{y}^{\ast}$ and go through $\y_\v$, so the direction of this ray is $\y_\v - \y^*$. To show the objective strictly decreases along this ray from $\y_\v$, it suffices to show that the objective value is strictly higher at $\y^*$ than at $\y_\v$. Evaluating at both points, we have:
\begin{align*}
\overline{\lambda_\v\c}^{\T} \y^* &= \lambda_\v\c^\T\x_0 + M+1; \\
\overline{\lambda_\v\c}^{\T} \y_\v &= \lambda_\v\c^\T\v + \w^\T\v. 
\end{align*}

 In Construction \ref{construction: Q tilde}, $M$ is defined to be sufficiently large:
\begin{align*}
  M &= 2\max(\{\mathbf{w}^{\intercal} \mathbf{x}: \mathbf{x} \in P\} \cup \{\lambda_{\mathbf{v}}\mathbf{c}^{\intercal}(\mathbf{v}-\mathbf{x}^{\ast}): \mathbf{v} \in U(P)\},0)\\
  &\geq \lambda_\v \c^\T(\v-\x_0) + \w^{\T}\v.
\end{align*}
It follows then that the objective value at $\y^*$ is strictly greater than at $\y_\v$. Since the feasible ray here on $\tilde{Q}$ moves away from $\y^*$ in the direction $\y_\v - \y^*$, the objective strictly decreases along it.

We have shown that the objective function strictly decreases along all local direction from $\y_\v$, so $\y_\v$ is the unique maximizer.
\end{proof}

To summarize, we have:

\begin{lem}
\label{lem:theseed}
For each $d \geq 2$, there exists a $(d+1)$-dimensional polyhedron $\tilde{Q}$ in $\mathbb{R}^{d+1}$and $\mathbf{c} \in \mathbb{R}^{d}$ such that
\begin{enumerate}
  \item $\tilde{Q}$ has $2d+1$ facets
  \item all outer normals of $\tilde{Q}$ point upward.
  \item the number of unique maximizers of $\overline{\lambda \mathbf{c}}^\T \y$ is at least $2^{d-1}$.
  \item there is a unique vertex $\v^*$ that maximizes $\e_{d+1}$.
\end{enumerate}
 In particular, $\tilde{Q}$ is a seed polyhedron in Definition \ref{def: seed}.
\end{lem}

\begin{proof}
Take the $(d+1)$-dimensional polyhedron $\tilde{Q}$ as constructed: it has $2d+1$ facets. Lemma \ref{lem:allpositive} shows that all outer normals point upward, and Lemma \ref{lem:liftupperhull} proves that each of the $\geq 2^{d-1}$ vertices in $U(P)$ is lifted to a unique maximizer of $\overline{\lambda\c}^\T \y$ over $\tilde{Q}$. Finally, since by Corollary \ref{cor:upperhull}, $P$ has a unique $\w$-maximal vertex, $\tilde{P}$ has a unique $\e_{d+1}$-maximizer by applying Lemma \ref{lem:liftupperhull} in the case where $\lambda_{\mathbf{v}} = 0$.

\end{proof}

\subsection{Broadening the Candidates}

 The first key observation for our construction is that the choice of $\mathbf{c}$ such that the number of maximizers of linear function $\overline{\lambda \c}^\T \x$ is at least $2^{d-1}$ is not unique. In fact, if one such vector exists, there is a full dimensional cone of vectors that also work. Recall that we defined a cone to be upward facing if each of its rays has positive final coordinate. 

\begin{lem}
\label{lem:broadening}
Let $Q \subseteq \mathbb{R}^{d+1}$ be a polyhedron with a unique vertex $\v^*$ maximizing the last coordinate. Suppose there exists $\mathbf{c} \in \mathbb{R}^{d}$ such that the ray $\r(\c) = \{(\lambda\mathbf{c},1): \lambda \geq 0\}$ intersects the interior of $\alpha$ upward facing normal cones of $Q$. Then there exists a $d$-dimensional simplicial cone $C \subseteq \mathbb{R}^{d}$ such that $\r(\w)$ intersects $\alpha$ upward facing normal cones of $Q$ for all $\mathbf{w} \in C \setminus \{\mathbf{0}\}$.

\end{lem}

\begin{proof}
Let $N_1, N_2, \dots, N_\alpha$ be the $\alpha$ upward facing cones whose interiors are intersected by $\r(\c)$. Since these interiors are open, there is sufficient local wiggle room to perturb the ray while still maintaining the intersections. 

Specifically, intersect all these cones with the hyperplane $y_{d+1}=1$ to get a polyhedral complex containing the ray $\r(\c)$. For each $i$, the intersection of $\r(\c)$ with the interior of $N_i$ forms an open line segment. Pick a point $\x_i$ in the interior of this segment. Since $\x_i \in \text{int}(N_i)$, we can pick a small open ball $B_i$ centered at $\x_i$ that remains entirely in $\text{int}(N_i)$. Let $R_i$ be the cone of all rays from the apex $\e_{d+1}$ to $B_i$. Each $R_i$ (excluding the apex $\e_{d+1}$) is open since $B_i$ is an open ball, and so is their intersection $R = \cap_{i=1}^\alpha R_i$. Furthermore, $R = \cap R_i$ is non-empty since it contains $\r(\c)$, hence it remains a $d$-dimensional cone, and by construction every ray in $R$ intersects the interiors of all of the $\alpha$ normal cones $N_1, \dots, N_\alpha$. We can simply choose $C$ to be any full dimensional simplicial cone in $R $.
\end{proof}

The cone could be arbitrarily narrow to begin with, so we apply an invertible linear transformation to stretch it and make that cone whatever we want. 

\begin{prop}
  Let $Q \subseteq \mathbb{R}^{d+1}$ be the polyhedron in Lemma \ref{lem:broadening}, and let $C'$ be \emph{any} $d$-dimensional simplicial cone in $\mathbb{R}^{d}$. There exists an invertible linear transformation $T:\mathbb{R}^{d+1} \to \mathbb{R}^{d+1}$ such that, over the polyhedron $Q' = T(Q)$, the vertex $T(\v^*)$ uniquely maximizes the last coordinate, and the ray $\r (\x)$ intersects $\alpha$ upward facing normal cones of $Q'$ for all $\x \in C'$.
\end{prop}

\begin{proof}
  Let $C$ be the simplicial cone of vectors, each intersecting $\alpha$ upward facing normal cones of $Q$ from Lemma \ref{lem:broadening}. Let $\mathbf{a}_{1}, \dots, \mathbf{a}_{d}$ be the ray generators of $C$, and let $\mathbf{a}_{1}', \dots, \mathbf{a}_{d}'$ be the generators of the target cone $C'$. Define a linear transformation $S:\mathbb{R}^{d+1} \to \mathbb{R}^{d+1}$ by $S(\e_{d+1}) = \e_{d+1}$, and $S(\mathbf{a}_{i},0) = (\mathbf{a}_{i}',0)$ for all $1\leq i\leq d$. Then $S(C) = C'$ ($C$, $C'$ are embedded in $\{y_{d+1} =0\} \subset \R^{d+1}$). 

Define $T:= (S^{-1})^\T$, and 
\[Q' := T(Q) = (S^{-1})^\T(Q).\] By Lemma \ref{lem:normalfantransform}, each normal cone of $Q'$ is precisely $S(N)$ where $N$ is a normal cone of $Q$. It follows that for any nonzero vector $\x \in C'$, its preimage is $S^{-1}(\x,0) \in C'$, and so the ray $\r(S^{-1}(\x))$ intersects the interior of a normal cone $N$ of $Q$. So the ray $\r(\x)$ intersects the interior of $S(N)$ for any normal cone $N$ that $\w$ intersects. Finally, since $S$ preserves the last coordinate, so does its inverse transpose $T$. The last coordinate of every vertex of $Q$ is preserved, so $T(\v^*)$ remain the maximizer in $Q'$, and each of the $\alpha$ normal cones remain upward facing.
\end{proof}

\begin{figure}
  \centering
  
\[\begin{tikzpicture}[scale = 1.25]
  \draw[<->] (-1.5,0) -- (1.5,0);
  \draw[<->] (0,-1.5) -- (0,1.5);
  
  
  \draw[thick, ->]
  (0,0) -- (1.5,1);
   \draw[thick, ->]
  (0,0) -- (1.25,1);
  \draw[thick, ->]
  (0,0) -- (-1.25,1);
  \draw[thick, ->]
  (0,0) -- (-1.5, 1);

\end{tikzpicture}
\hskip.2in
\begin{tikzpicture}[scale = 1.25]
  \draw[<->] (-1.5,0) -- (1.5,0);
  \draw[<->] (0,-1.5) -- (0,1.5);
  
  
  \draw[thick, ->]
  (0,0) -- (1.5,1);
   \draw[thick, ->]
  (0,0) -- (1.25,1);
  \draw[thick, ->]
  (0,0) -- (-1.25,1);
  \draw[thick, ->]
  (0,0) -- (-1.5, 1);
  \draw[thick, ->, blue, dashed]
  (0,0) -- (90:1);

  \draw (0, 1) node[red, circle, fill, inner sep = 20pt, opacity = .3] {};
\end{tikzpicture} \hskip.2in
\begin{tikzpicture}[scale = 1.25]
  \draw[<->] (-1.5,0) -- (1.5,0);
  \draw[<->] (0,-1.5) -- (0,1.5);
  
  
  \draw[thick, ->]
  (0,0) -- (.5,1);
   \draw[thick, ->]
  (0,0) -- (.42,1);
  \draw[thick, ->]
  (0,0) -- (-.42,1);
  \draw[thick, ->]
  (0,0) -- (-.5, 1);
  \draw[thick, ->, blue, dashed]
  (0,0) -- (90:1);

  \draw (0, 1) node[red, circle, fill, inner sep = 20pt, opacity = .3] {};
\end{tikzpicture}\] 
  \caption{Pictured on the left is a normal fan of an unbounded polygon with $\e_{2}$ uniquely optimized at a vertex and all normals upward facing. The middle depicts an epsilon ball around $\e_{2}$. The right shows the result after applying a linear transformation rescaling the first coordinates to get each normal to be within $\varepsilon$ of $\e_{2}$.}
  \label{fig:Compress}
\end{figure}
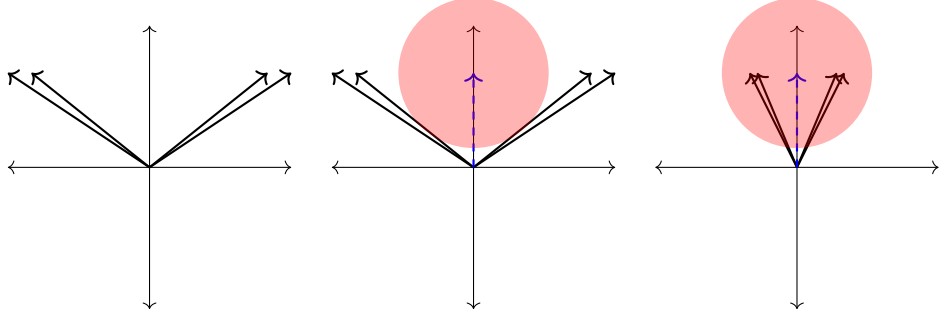

\subsection{Shrinking the Fan} 
We have already constructed a polyhedron such that there is an arbitrarily large simplicial cone of directions that hit exponentially many of its normal cones. For our purposes we require every orthogonal direction to be like this, which is stronger. Our strategy is to glue together copies of the same polyhedron. However, one cannot do this directly. Instead, we will shrink and insert the normal fans inside one another in a matryoshka doll like fashion. This next lemma shows that we may flatten the polyhedron, which is equivalent to squeezing its normal fan.

\begin{lem}
\label{lem: new shrinking}
Let $P \subset \R^{d+1}$ be a polyhedron with vertex $\v^*$ that uniquely maximizes the last coordinate, and assume that all the outer normals of $P$ point upward. Then, for any arbitrarily small $\varepsilon > 0$, there exists a polyhedron $Q$ combinatorially equivalent to $P$ satisfying the following conditions:
\begin{enumerate}
  \item All generating rays of the normal fan of $Q$ are within an $\varepsilon$-neighborhood of the vertical vector $\e_{d+1}$.
  \item $Q$ retains the unique maximizer of the last coordinate, and outer normals remain up-pointing.
  \item For any $\x \in \mathbb{R}^{d}$, if the horizontal ray $\r(\x) = \{(\lambda \x,1)\; : \; \lambda\geq 0\}$ intersects the interiors of $\alpha$ maximal normal cones of $P$, then it intersects the interiors of the corresponding $\alpha$ maximal normal cones in $Q$.
\end{enumerate}
\end{lem}

\begin{proof}
Let $T$ be the flattening transformation from Corollary \ref{cor: squeezing normal fan} given by the matrix $\text{diag}(1, \dots, 1, c)$ for some $0 < c < 1$. Also define $Q= T(P)$. This map presses $P$ towards the horizontal hyperplane $y_{d+1}=0$. In the dual space, it ($(T^{-1})^\T$) squeezes the normal cones towards the vertical axis $\e_{d+1}$. By choosing $c$ sufficiently small, Corollary \ref{cor: squeezing normal fan} implies Condition (1) holds. Since $c>0$, the map $T$ strictly preserves the order of vertices in the last coordinate, so $T(\v^*)$ remains the vertex with maximal last coordinate in $Q$.

To verify the normal vectors, consider the dual space of $Q$. The normal cones of $Q$ are given by applying $S =(T^{-1})^{\T}$ to the normal cones of $P$, where the matrix of $S$ is $\text{diag}(1,1,\dots, 1, 1/c)$. Let $\n = (n_1, \dots, n_{d+1})$ be an outer normal vector of a facet of $P$, the corresponding normal vector of $Q$ is
\[ S(\n) = (n_1, \dots, n_d, \frac{1}{c}n_{d+1}).\]
It is upward pointing since both $c>0$ and $n_{d+1}>0$, so Condition 2 is satisfied.

We now track the ray $\r(\x)$. It intersects the cone $S(N)$ if and only if the inverse-transformed ray $S^{-1}(\r(\x))$ intersects the original normal cone $N$ of $P$. The matrix of $S^{-1} = T^\T$ is $\text{diag}(1, \dots, 1, c)$. Apply it to the ray we get
\[S^{-1}(\r (\x)) =\{(\lambda \x, c) :\; \lambda\geq 0 \},\] 
which is a parallel ray on the horizontal hyperplane $y_{d+1} = c$. Since each normal cone is invariant under positive scaling, the point $\y = (\lambda \x, c)$ is in $N$ if and only if $\{\frac{1}{c}\y = (\frac{\lambda}{c} \x, 1)\;: \; c>0\}$ also belong to $N$ (for all $c>0$, $\lambda\geq 0$. The set of points is exactly $\r(\x)$ itself! Therefore, the original ray $\r(\x)$ intersects $S(N)$ if and only if it intersects $N$. This preserves the exact sequence of maximal cones crossed, satisfying Condition 3.

Geometrically, the rays $\r(\x)$ and $S^{-1}(r(\x))$ lie in the same (quadrant of) $2$-dimensional subspace of $\R^{d+1}$ spanned by $\e_{d+1}$ and $(\x,0)$. We are slicing $\alpha$ upward-pointing normal cones with two parallel horizontal hyperplanes of $\R^{d+1}$. Because all normal cones are anchored at the origin, slicing the normal fan within this subspace at two different positive heights only scales the intersection pattern from the origin, therefore the two rays cross the exact same sequence of maximal cones.
\end{proof}

\subsection{Common Refinement without Minkowski Sums}

Given a pair of polyhedra $P$ and $Q$, one of the most well-known tools to build a new one whose normal fan refines both $P$ and $Q$ is to take their Minkowski sum $P+Q$ (see, for example, Proposition 7.12 in \cite{LecturesonPolytopes}). The Minkowski sum operation is a powerful tool for building complicated polyhedral fans. However, in general, the new polyhedron $P+Q$ can have exponentially more facets than $P$ and $Q$. For example, one can take a Minkowski sum of two sufficiently generic linear transformations of a cube. The result would be a generic zonotope with $2d$ generators and therefore have $2\binom{2d}{d-1}$ facets.

This is a problem for our complexity results as we want to build polytopes with relatively few facets. We overcome this challenge by introducing the following novel lemma that is a key tool for our construction method. See Figure \ref{fig:refinement} for a visual description of the idea.

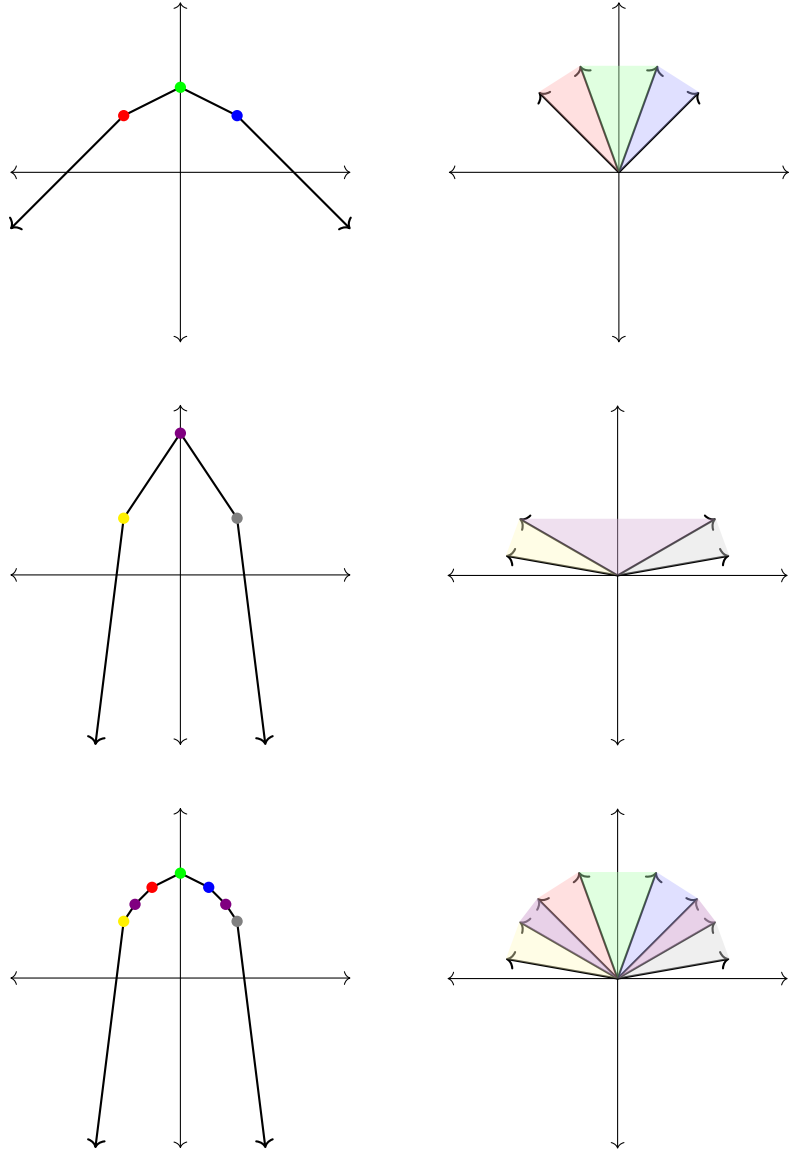
\begin{figure}
  \centering
  \[
\begin{tikzpicture}[scale =1.5]
  \draw[<->] (-1.5,0) -- (1.5,0);
  \draw[<->] (0,-1.5) -- (0,1.5);
  \draw[thick] (-1,0) -- (-.5,.5) -- (0,.75) -- (.5,.5) -- (1,0);
  \draw[thick, ->] (1,0) -- (1.5,-.5);
  \draw[thick, ->] (-1,0) -- (-1.5,-.5);
  \draw (-.5, .5) node[red, circle, fill, inner sep = 1.5pt] {};
  \draw (0, .75) node[green, circle, fill, inner sep = 1.5pt] {};
  \draw (.5, .5) node[blue, circle, fill, inner sep = 1.5pt] {};
\end{tikzpicture}
\hskip.5in
\begin{tikzpicture}[scale = 1.5]
  \draw[<->] (-1.5,0) -- (1.5,0);
  \draw[<->] (0,-1.5) -- (0,1.5);
  
  
  \draw[thick, ->]
  (0,0) -- (45:1);
   \draw[thick, ->]
  (0,0) -- (70:1);
  \draw[thick, ->]
  (0,0) -- (135:1);
  \draw[thick, ->]
  (0,0) -- (110:1);
  \fill[red!30, opacity=0.4] (0,0) -- (135:1) -- (110:1) -- cycle;
  
  \fill[green!30, opacity=0.4] (0,0) -- (110:1) -- (70:1)--cycle;
  
  \fill[blue!30, opacity=0.4] (0,0) -- (70:1) -- (45:1)--cycle;
\end{tikzpicture}\]

 \[
\begin{tikzpicture}[scale =1.5]
  \draw[<->] (-1.5,0) -- (1.5,0);
  \draw[<->] (0,-1.5) -- (0,1.5);
  \draw[thick](-.5,.5) -- (0,1.25) -- (.5,.5);
  \draw[thick, ->] (.5,.5) -- (.75,-1.5);
  \draw[thick, ->] (-.5,.5) -- (-.75,-1.5);
  \draw (-.5, .5) node[yellow, circle, fill, inner sep = 1.5pt] {};
  \draw (0, 1.25) node[violet, circle, fill, inner sep = 1.5pt] {};
  \draw (.5, .5) node[gray, circle, fill, inner sep = 1.5pt] {};

\end{tikzpicture}
\hskip.5in
\begin{tikzpicture}[scale = 1.5]
  \draw[<->] (-1.5,0) -- (1.5,0);
  \draw[<->] (0,-1.5) -- (0,1.5);
  
  
  \draw[thick, ->]
  (0,0) -- (10:1);
   \draw[thick, ->]
  (0,0) -- (30:1);
  \draw[thick, ->]
  (0,0) -- (170:1);
  \draw[thick, ->]
  (0,0) -- (150:1);
  \fill[gray!30, opacity=0.4] (0,0) -- (10:1) -- (30:1) -- cycle;
  
  \fill[violet!30, opacity=0.4] (0,0) -- (30:1) -- (150:1)--cycle;
  
  \fill[yellow!30, opacity=0.4] (0,0) -- (150:1) -- (170:1)--cycle;
\end{tikzpicture}\]

 \[
\begin{tikzpicture}[scale =1.5]
  \draw[<->] (-1.5,0) -- (1.5,0);
  \draw[<->] (0,-1.5) -- (0,1.5);
  \draw[thick, ->] (.5,.5) -- (.75,-1.5);
  \draw[thick, ->] (-.5,.5) -- (-.75,-1.5);

  \draw[thick] (-.5,.5) -- (-.4,.65);
  \draw[thick] (.5,.5) -- (.4,.65);
  \draw[thick] (-.4,.65) --(-.25,.25+.55) -- (0,.375+.55) -- (.25,.25+.55)--(.4, .65);
  \draw (-.5, .5) node[yellow, circle, fill, inner sep = 1.5pt] {};
  \draw (.5, .5) node[gray, circle, fill, inner sep = 1.5pt] {};
  
  \draw (-.25, .8) node[red, circle, fill, inner sep = 1.5pt] {};
  \draw (0, .925) node[green, circle, fill, inner sep = 1.5pt] {};
  \draw (.25, .8) node[blue, circle, fill, inner sep = 1.5pt] {};
  \draw (-.4, .65) node[violet, circle, fill, inner sep = 1.5pt] {};
  \draw (.4, .65) node[violet, circle, fill, inner sep = 1.5pt] {};
\end{tikzpicture}
\hskip.5in
\begin{tikzpicture}[scale = 1.5]
  \draw[<->] (-1.5,0) -- (1.5,0);
  \draw[<->] (0,-1.5) -- (0,1.5);
  
  
  \draw[thick, ->]
  (0,0) -- (10:1);
   \draw[thick, ->]
  (0,0) -- (30:1);
  \draw[thick, ->]
  (0,0) -- (170:1);
  \draw[thick, ->]
  (0,0) -- (150:1);

  \draw[thick, ->]
  (0,0) -- (45:1);
   \draw[thick, ->]
  (0,0) -- (70:1);
  \draw[thick, ->]
  (0,0) -- (135:1);
  \draw[thick, ->]
  (0,0) -- (110:1);
  \fill[gray!30, opacity=0.4] (0,0) -- (10:1) -- (30:1) -- cycle;
  
  \fill[violet!30, opacity=0.6] (0,0) -- (30:1) -- (45:1)--cycle;

  \fill[violet!30, opacity=0.6] (0,0) -- (150:1) -- (135:1)--cycle;
  
  \fill[yellow!30, opacity=0.4] (0,0) -- (150:1) -- (170:1)--cycle;

  \fill[red!30, opacity=0.4] (0,0) -- (135:1) -- (110:1) -- cycle;
  
  \fill[green!30, opacity=0.4] (0,0) -- (110:1) -- (70:1)--cycle;
  
  \fill[blue!30, opacity=0.4] (0,0) -- (70:1) -- (45:1)--cycle;
\end{tikzpicture}\]
  \caption{The left depicts three unbounded polygons with all upward facing facets normals, and the right depicts their normal fans. The normal fan of the first polygon fits inside the part of the normal cone of the second containing the vector $\e_{2}$ as in the assumptions for Lemma \ref{lem:refinement}. By that lemma, there is then another polygon obtained shrinking and translating the first polygon and intersecting it with the second such that the normal fan of the resulting polygon refines both of their normal fans. The third polygon and normal fan depict such an intersection and refinement. }
  \label{fig:refinement}
\end{figure}

\begin{lem}
\label{lem:refinement}
Let $P\subseteq \mathbb{R}^{d}$ be a full-dimensional polyhedron with $m$ facets and a vertex $\v$ with normal cone $N_\v$. Let $Q\subseteq \mathbb{R}^{d}$ be a full-dimensional polyhedron with at least one vertex and exactly $q$ facets, satisfying that all of its facet outer normals are contained within the interior of $N_\v$. Then there exists a polyhedron $R$ with exactly $m + q$ facets such that $\mathcal{N}(R) \supseteq (\mathcal{N}(P) \cup \mathcal{N}(Q)) \setminus N_{\v}$. 
\end{lem}

\begin{proof}
  We construct the new polyhedron $R$ by cutting the vertex $\mathbf{v}$ off from $P$ and replacing it with a $Q$-cap that is a rescaled and translated copy of $Q$. 
  
First consider the geometric setting for some intuition. By polar duality, the support of $\mathcal{N}(Q)$ being inside the interior of the normal cone $(N_\mathbf{v})$ implies that the tangent cone of $P$ at $\mathbf{v}$ is strictly contained inside the interior of the recession cone of $Q$, which is the cone of all unbounded directions inside $Q$. Geometrically, this means that $Q$ opens strictly wider than $P$ locally at $\mathbf{v}$. This is what allows us to truncate $\mathbf{v}$ using $Q$ without interrupting the rest of $P$: we will translate $Q$ past $\v$ in order to exclude $\v$ from its intersection with $P$, and rescale it to be sufficiently small to prevent its facets from being chopped off by the facets of $P$. The specific details of our construction are as follows.
  
  Let $Q = \{\x \in\mathbb{R}^{d}: \; A\x\leq \mathbf{b} \}$ be the inequality description of $Q$ without redundancy. By our assumption, each row vector (outer normal) $\mathbf{A}_i$ lies in the interior of $N_\v$. By Proposition \ref{prop: maximizer vs normal cone}, linear function $\mathbf{A}_i^\T\x$ is uniquely maximized over $P$ at the vertex $\v$. Therefore there exists a small radius $\delta>0$ such that 
  \[
  \mathbf{A}_i^\T \x > \mathbf{A}_i^\T \v'
  \]
  for all $\x$ satisfying $||\x-\v||<\delta$ and any other vertex $\v'$ of $P$.

 Pick a point $\u^*$ strictly in the interior of $P$ at distance at most $\delta/2$ from $\v$. Since $\u^{\ast}$ is on the interior of $P$, there exists an open ball of radius $\gamma$ around $\u^{\ast}$ also contained in the interior of $P$. Let $\u$ be a vertex of $Q$. Rescale $Q$ using a tiny factor such that all remaining vertices are clustered within the distance of $\varepsilon >0$ of $\u$, where $0<\varepsilon< \min(\delta/2, \gamma)$. Call the new polyhedron we obtained $Q_{\text{mini}}$. Next, we translate $Q_{\text{mini}}$ so that $\u$ is mapped to $\u^*$. Call the translated copy $Q_{\text{cap}}$, i.e.,
  \[Q_{\text{cap}} = Q_{\text{mini}} + (\u^*-\u).\]
  Our final step is to put this cap on $P$ by defining
  \[R : = P \cap Q_{\text{cap}}.\]
  
  Let's verify the construction of $R$ satisfies the required conditions. Since all of the vertices of $Q_{\text{cap}}$ are inside an $\varepsilon$-ball centered at $\u^*$ and $\varepsilon < \gamma$, where $\gamma$ is the radius of a ball contained around $\u^{\ast}$ in the strict interior of $P$, each vertex of $Q_{\text{cap}}$ is in $R$ and therefore feasible.  Any inequality that was tight at them in $Q_{\text{cap}}$ before remains tight, and so they remain basic feasible and therefore vertices.

  Each vertex of $Q_{\text{cap}}$ is contained in an $\varepsilon$ ball around $\mathbf{u}^{\ast}$. Hence, it is of distance at most $\varepsilon <\delta/2$ from $\mathbf{u}^{\ast}$. By construction $\mathbf{u}^{\ast}$ is at most $\delta/2$ away from $\mathbf{v}$, so each vertex of $Q_{\text{cap}}$ is of distance at most $\delta$ from $\mathbf{v}$. Hence, by our choice of $\delta$, for any outer normal $\mathbf{A}_{i}$ of $Q_{\text{cap}}$, $A_{i}^{\intercal}\mathbf{x}$ is larger at any vertex of $Q_{\text{cap}}$ than any vertex of $P$ other than $\mathbf{v}^{\ast}$. Thus, any vertex of $P$ other than $\mathbf{v}^{\ast}$ remains feasible and therefore a basic feasible solution. Hence, each vertex of $P$ remains a vertex of $R$ as well. The set of tight inequalities at each vertex does not change, and so their normal cones are also preserved meaning that the normal fan of $R$ refines that of $P$ and $Q$. Thus, every normal cone of a vertex of $P$ and a vertex of $Q$ other than $\mathbf{v}$ is a normal cone of a vertex of $R$. Thus, $\mathcal{N}(R) \supseteq (\mathcal{N}(P) \cup \mathcal{N}(Q)) \setminus N_{\v}$.

  The intersection of two polyhedra also does not introduce unexpected new facets. Therefore the number of facets of $R$ is the sum of the numbers of facets of $P$ and $Q$.
\end{proof}

Note that, while we know the normal fan of $R$ is a refinement of the two fans, we do not have a complete description of its fan from the construction. This construction is coarse in that sense meaning that we do not actually see globally what the resulting polytope looks like. It could be that this already yields a counterexample to the polynomial Hirsch conjecture. However, this coarse information suffices for our arguments. 

\subsection{Building the Example}

Finally, to implement our construction, we start with the seed polyhedron. We blow it up, so that it has a cone of directions yielding exponential sized projections so large that $d+1$ rotated copies of it cover the entire space. Then we take $d+1$ copies of these blown up seeds and progressively shrink and glue them together to finally build our bad example.

\begin{theorem}\label{thm: unbounded main}
For each $d \geq 3$, there exists a $d+1$ dimensional polyhedron $P$ with $(2d+1)(d+1)$ facets such that for any vector $\c \in \mathbb{R}^{d}\backslash \mathbf{0}$, the ray $\r(\c) = \{(\lambda\c,1):\lambda\geq 0 \}$ crosses through the interior of at least $2^{d-1}$ normal cones of $P$.

\end{theorem}
\begin{proof}
By Lemma \ref{lem:theseed}, there exists a $d+1$ dimensional unbounded polyhedron $Q$ with $2d+1$ facets and all upward pointing outer normals, such that $\e_{d+1}$ is uniquely maximized at one vertex $\v^*$, and there exists $\mathbf{c}^* \in \mathbb{R}^{d}$ such that the ray $r (\mathbf{c}^*) = \{(\lambda\c^*, 1): \lambda \geq 0\}$ intersects $2^{d-1}$ upward facing normal cones of vertices of $Q$.

Let $C_{1}, C_{2}, \dots, C_{d+1}$ be the maximal cones of the normal fan of any $d$-dimensional simplex in $\mathbb{R}^{d}$ embedded in $\mathbb{R}^{d+1}$ via appending a $0$ in their last coordinate. By Lemma \ref{lem:broadening}, we may apply linear transformations to arrive at polyhedra $Q_{1}, Q_{2}, \dots, Q_{d+1}$ such that for each $i$, the set of rays $\r(\mathbf{c})$ such that $\mathbf{c} \in C_{i}$ all intersect $2^{d-1}$ many upward facing maximal normal cones of $Q_{i}$. Furthermore, by Lemma \ref{lem:broadening}, this transformation preserves that all normal vectors of facets point upward.

Set $R_{1} = Q_{1}$. Then $R_{1}$ has a unique maximizer of $\e_{d+1}$. We build $R_{i}$ recursively. Assume by induction that $R_{i-1}$ has a unique $\e_{d+1}$-maximizer vertex $\mathbf{v}$. By Proposition \ref{prop: maximizer vs normal cone}, $\e_{d+1}$ is in the interior of the normal cone $N_\v$. Then there exists $\varepsilon > 0$ sufficiently small so that any vector of distance at most $\varepsilon$ from $\e_{d+1}$ is also in $\text{int}(N_\v)$. Since each inequality defining $Q_{i}$ has positive last coordinate, by Lemma \ref{lem: new shrinking}, there is a map yielding $Q_{i}'$, a flattened copy of $Q_{i}$ such that all of its outer normals are within $\varepsilon$ of $\e_{d+1}$, and there is still a unique $\e_{d+1}$-maximal vertex.

Next, we construct $R_{i}$. Since the support of the normal fan of $Q_{i}'$ is contained within an $\varepsilon$-ball centered at $\e_{d+1}$, it is inside the normal cone $N_{\v}$ of polyhedron $R_{i-1}$ by our choice of $\varepsilon$. Thus, they satisfy the assumptions of Lemma \ref{lem:refinement} and therefore can create a polyhedron $R_{i}$ such that the normal fan of $R_{i}$ refines both the normal fans of $Q_{i}'$ and $R_{i-1}$. Applying this process $d$ times yields our final $R_{d+1}$. Since each $Q_{i}$ has $2d+1$ facets, the total number of facets of the resulting polyhedron is $(2d+1)(d+1)$ by the additivity of the number of facets in the refinement from Lemma \ref{lem:refinement}.
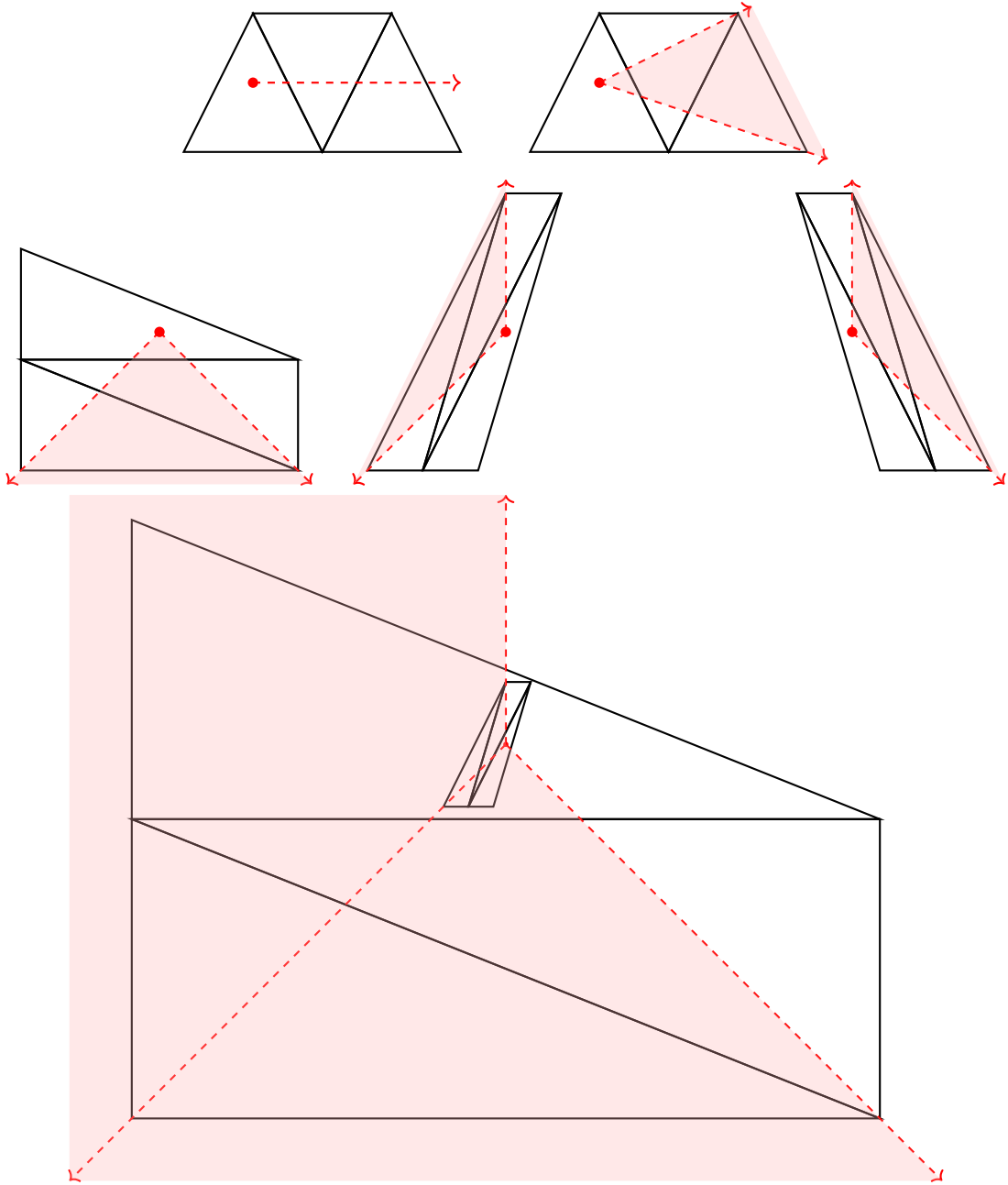
\begin{figure}
  \centering
  \[
\begin{tikzpicture}
    \draw[black, thick] (-1,-1) -- (0,1) -- (1,-1) -- cycle;
    \draw[black, thick] (0,1) -- (1,-1) -- (2,1) -- cycle;
    \draw[black, thick] (1,-1) -- (2,1) -- (3,-1) -- cycle;

    \draw[red, thick, ->, dashed] (0,0) -- (3,0);
    
    \draw (0, 0) node[red, circle, fill, inner sep = 1.5pt] {};

    \begin{scope}[shift={(5,0)}]
    \draw[black, thick] (-1,-1) -- (0,1) -- (1,-1) -- cycle;
    \draw[black, thick] (0,1) -- (1,-1) -- (2,1) -- cycle;
    \draw[black, thick] (1,-1) -- (2,1) -- (3,-1) -- cycle;

    \draw[red, thick, ->, dashed] (0,0) -- (2.2,1.1);

     \draw[red, thick, ->, dashed] (0,0) -- (3.3,-1.1);

     \fill[red!30, opacity=.3] (0,0) -- (2.2,1.1) -- (3.3,-1.1) -- cycle;
    
    \draw (0, 0) node[red, circle, fill, inner sep = 1.5pt] {};
    \end{scope}




    
  \end{tikzpicture}
  \]
\[\begin{tikzpicture}
  \begin{scope}[cm={2/5,0,-4/5,10/5,(10,0)}]
    \draw[black, thick] (-1,-1) -- (0,1) -- (1,-1) -- cycle;
    \draw[black, thick] (0,1) -- (1,-1) -- (2,1) -- cycle;
    \draw[black, thick] (1,-1) -- (2,1) -- (3,-1) -- cycle;

    \draw[red, thick, ->, dashed] (0,0) -- (2.2,1.1);

     \draw[red, thick, ->, dashed] (0,0) -- (3.3,-1.1);

     \fill[red!30, opacity=.3] (0,0) -- (2.2,1.1) -- (3.3,-1.1) -- cycle;
    
    \draw (0, 0) node[red, circle, fill, inner sep = 1.5pt] {};
    \end{scope}
    \begin{scope}[cm={-2/5,0,4/5,10/5,(5,0)}]
    \draw[black, thick] (-1,-1) -- (0,1) -- (1,-1) -- cycle;
    \draw[black, thick] (0,1) -- (1,-1) -- (2,1) -- cycle;
    \draw[black, thick] (1,-1) -- (2,1) -- (3,-1) -- cycle;

    \draw[red, thick, ->, dashed] (0,0) -- (2.2,1.1);

     \draw[red, thick, ->, dashed] (0,0) -- (3.3,-1.1);

     \fill[red!30, opacity=.3] (0,0) -- (2.2,1.1) -- (3.3,-1.1) -- cycle;
    
    \draw (0, 0) node[red, circle, fill, inner sep = 1.5pt] {};
    \end{scope}

    \begin{scope}[cm={0,-4/5,10/5,-2/5,(0,0)}]
    \draw[black, thick] (-1,-1) -- (0,1) -- (1,-1) -- cycle;
    \draw[black, thick] (0,1) -- (1,-1) -- (2,1) -- cycle;
    \draw[black, thick] (1,-1) -- (2,1) -- (3,-1) -- cycle;

    \draw[red, thick, ->, dashed] (0,0) -- (2.2,1.1);

     \draw[red, thick, ->, dashed] (0,0) -- (3.3,-1.1);

     \fill[red!30, opacity=.3] (0,0) -- (2.2,1.1) -- (3.3,-1.1) -- cycle;
    
    \draw (0, 0) node[red, circle, fill, inner sep = 1.5pt] {};
    \end{scope}
\end{tikzpicture}\]

\begin{tikzpicture}[scale =.9]
   \draw (0, 0) node[red, circle, fill, inner sep = .75pt] {};



    

  \begin{scope}[cm={-1/5,0,2/5,5/5,(0,0)}]
    \draw[black, thick] (-1,-1) -- (0,1) -- (1,-1) -- cycle;
    \draw[black, thick] (0,1) -- (1,-1) -- (2,1) -- cycle;
    \draw[black, thick] (1,-1) -- (2,1) -- (3,-1) -- cycle;



    
    \end{scope}
    
    \begin{scope}[cm={0,-12/5,30/5,-6/5,(0,0)}]
    \draw[black, thick] (-1,-1) -- (0,1) -- (1,-1) -- cycle;
    \draw[black, thick] (0,1) -- (1,-1) -- (2,1) -- cycle;
    \draw[black, thick] (1,-1) -- (2,1) -- (3,-1) -- cycle;
    \end{scope}
    \draw[red, thick, dashed, ->] (0,0) -- (-7,-7);
    \draw[red, thick, dashed, ->] (0,0) -- (7,-7);
    \draw[red, thick, dashed, ->] (0,0) -- (0,4);
     \fill[red!30, opacity=.3] (0,0) -- (7,-7) -- (-7, -7) -- (-7,4) -- (0,4) -- cycle;

\end{tikzpicture}





    


  \caption{The images depict three cones of the normal fan of five $3$-D polyhedra with upper facing normals intersected with the hyper-plane $(0,0,1)\cdot\mathbf{x} = 1$. In the top, there is a ray from $\e_{3}$ intersecting the interiors of all of the cones. As shown in the proof of Lemma \ref{lem:broadening}, there must therefore be a full dimensional simplicial cone with vertex $\e_{3}$ such that every ray of that cone intersects each of those cones on their interior. Such a cone is depicted in the top right image. Then, by Lemma \ref{lem:broadening}, one may apply a linear transformation to the normal fan such that this cone is whatever cone we choose. In the proof of Theorem \ref{thm: unbounded main}, we apply a transformation to find polyhedra $Q_{1}, Q_{2},$ and $Q_{3}$ such that the set of rays intersecting the three cones covers all of space. Finally, the bottom figure shows the shrinking of $Q_{2}$ to be inside $Q_{1}$. The set of rays intersecting at least three cones then contains the union of the set from $Q_{1}$ and $Q_{2}$. }
  \label{fig:ProofIdea}
\end{figure}

By construction, the normal fan of $R_{d+1}$ contains each cone in the normal fans of $Q_{1}, Q_{2}', Q_{3}', \dots, Q_{d+1}'$ excluding their $\e_{d+1}$-maximal vertices, where each $Q_{i}'$ is obtained by rescaling from Lemma \ref{lem: new shrinking}. Let $\c \in \mathbb{R}^{d}\setminus \mathbf{0}$. Since $C_1, \dots, C_{d+1}$ are the maximal cones of the normal fan of a simplex, they cover $\R^d$. Therefore $\c \in C_{i}$ for some $i \in [d+1]$. Then $\r (\c)$ intersects $2^{d-1}$ many cones of $Q_{i}$, and by Lemma \ref{lem: new shrinking}, it also intersects $2^{d-1}$ many normal cones of $Q_{i}'$. Since the normal fan of $R_{d+1}$ refines $Q_{i}'$ excluding the normal cone of its $\e_{d+1}$-maximal vertex, $\r(\c)$ must intersect the interior of at least $2^{d-1}$ normal cones of $R_{d+1}$. Hence, for all $\c\in \R^d\backslash \mathbf{0}$, the ray $\r(\c)$ intersects the interior of at least $2^{d-1}$ normal cones, as desired.

\end{proof}

Notice that the polyhedron $P$ constructed here must be unbounded, since the condition that all outer normal vectors are upward pointing is preserved throughout the construction. We can make it bounded by adding a new facet down below.

\begin{cor}\label{cor: bounding from below}
  For each $d \geq 3$, there exists a $d+1$ dimensional polytope $P'$ with $(2d+1)(d+1) + 1$ facets such that for any vector $\c \in \mathbb{R}^{d}\backslash \mathbf{0}$, the ray $\r(\c) = \{(\lambda\c,1):\lambda\geq 0 \}$ crosses through the interior of at least $2^{d-1}$ of its normal cones.
\end{cor}

\begin{proof}
  Let $P$ be the polyhedron constructed in Theorem \ref{thm: unbounded main}. Let $H = \{\y\in \R^{d+1}: y_{d+1} = -N\}$ be a horizontal hyperplane whose last coordinate, $-N$, is strictly smaller than all the vertices of $P$. Let $P' = P\cap H^+$, where $H^+$ is the half space given by $y_{d+1} \geq - N $. By construction, $P'$ has exactly one more facet than $P$, which is the bottom one supported by $H$.
  
  We now show that $P'$ is bounded by showing the ray $\lambda \u$ escapes $P'$ for any direction $\u\in \R^{d+1}\backslash \mathbf{0}$. If $u_{d+1}<0$, than $\lambda \u$ must eventually crosses $H$ and escape $P'$. If $u_{d+1} \geq 0$, for $\lambda\u$ to stay inside $P'$ as $\lambda$ goes to infinity, it must satisfy that $\u ^\T\n <0$ for all facet outer normals $\n$ of $P'$. Since there exists a full-dimensional normal cone that contains $\e_{d+1}$ in its interior, there must exists an upward pointing outer normal vector $\n$ such that $\u ^\T\n >0$, hence $\lambda \u$ must eventually violate that facet inequality, and escape $P'$. We have shown that $P'$ does not contain a ray that goes to infinity, therefore it is bounded.
\end{proof}

Translating this result into the language of coherent monotone paths yields the proof of our main result.

\begin{proof}[Proof of Theorem \ref{thm:main}]
Take the polytope $P'$ and objective $\e_{d+1}$ from Corollary \ref{cor: bounding from below}. Let $\mathbf{c} \in \mathbb{R}^{d+1}$ be linearly independent of $\e_{d+1}$. Then, by Proposition \ref{prop: orthogonality}, the line $\{\mathbf{c} + \lambda \e_{d+1}: \lambda \in \mathbb{R}\} = \{\mathbf{c}' + \lambda \e_{d+1}: \lambda \in \mathbb{R}\}$ for some $\mathbf{c}'$ orthogonal to $\e_{d+1}$. In particular, the last coordinate of $\mathbf{c}'$ is equal to $0$. Thus, $\{\mathbf{c}' + \lambda \e_{d+1}: \lambda \in \mathbb{R}\} = \{(\mathbf{w},\lambda): \lambda \in \mathbb{R}\}$ for $\mathbf{w}$ the restriction of $\mathbf{c}'$ to its first $d$ coordinates. Note that 
\[\{(\mathbf{w},\lambda): \lambda \in \mathbb{R}\} \supseteq \{(\mathbf{w},\lambda): \lambda > 0\}.\]
Up to rescaling each point by a factor $\gamma = 1/\lambda$, this line is $\{(\gamma \mathbf{w},1): \gamma > 0\}$. In particular, since rescaling does not change the set of cones intersected, the line $\{\mathbf{c} + \lambda \e_{d+1}: \lambda \in \mathbb{R}\}$ intersects all of the normal cones of $P'$ that the ray $\{(\gamma \mathbf{w},1): \gamma > 0\}$ intersects. 

Since $P'$ satisfies the statement of Corollary \ref{cor: bounding from below}, that ray intersects at least $2^{d-1}$ normal cones of $P'$. Thus, by Proposition \ref{prop: coherentpath}, all coherent monotone paths on $P'$ with respect to $\e_{d+1}$ contain at least $2^{d-1}$ vertices and therefore have length at least $2^{d-1}-1$. 
\end{proof}

\subsection{Proof of applications}

We now prove Corollaries \ref{cor: shadowsimplex}, \ref{cor:stabbing} and \ref{cor:parametric optimization}. 

\subsubsection{Shadow Simplex Method} As noted in the introduction, the shadow simplex path from a vector $\mathbf{w}$ to a vector $\mathbf{c}$ consists of the set of unique optima of $\mathbf{w} + \lambda \mathbf{c}$ such that $\lambda \geq 0$. Sometimes it is phrased as taking all convex combinations of objectives between $\mathbf{c}$ and $\mathbf{w}$, but up to rescaling, this is equivalent. Unless $\mathbf{c}$ and $\mathbf{w}$ are sufficiently generic, there may be a need to break ties for choosing the path. Regardless of how these ties are broken, the length of the shadow simplex path is lower bounded by the number of normal cones intersected through their interior by the corresponding ray. Hence, our lower bound applies even to that case. 

\begin{proof}[Proof of Corollary \ref{cor: shadowsimplex}]
Take the polytope $P'$ and objective $\e_{d+1}$ from Corollary \ref{cor: bounding from below}. Then for any $\mathbf{c}$ orthogonal to $\e_{d+1}$, $\mathbf{c} = (\mathbf{c}',0)$. Thus, $\{(\lambda \mathbf{c}',1): \lambda \geq 0\} = \{\e_{d+1} + \lambda \mathbf{c}': \lambda \geq 0\}$. By the statement of Corollary \ref{cor: bounding from below}, this ray intersects $2^{d-1}$ many normal cones, so by Proposition \ref{prop: maximizer vs normal cone}, the set of maximizers of that ray is at least $2^{d-1}$. Hence, the corresponding shadow simplex path is of length at least $2^{d-1}-1$ as desired. 
\end{proof}

\subsubsection{Stabbing cells in regular triangulations}

For our application to stabbing subdivisions, we need to recall some definitions. Let $V= \{ \p_1, \p_2, \dots, \p_n\}$ be a finite set of points in $\R^d$ such that $P = \conv (V)$ is a $d$-dimensional polytope. A \textbf{subdivision} of $P$ (or of the point configuration $V$) is a collection of smaller $d$-dimensional polytopes $\mathcal{S} = \{P_1, \dots, P_s \}$ whose union is $P$ (each $P_i$ is called a \textbf{cell} in $\mathcal{S}$), and such that the intersection $P_i\cap P_j$ of any two cells is a single face of both or empty. If all the cells in $\mathcal{S}$ are simplices, then $\mathcal{S}$ is a \textbf{triangulation} of $P$.

A subdivision $\mathcal{S}$ is called \textbf{regular} if it can be constructed by a lifting of $P$. Specifically, given a weight function $w: \R^d \to \R$, we defined a $w$-induced lifting map $T_w: \R^d \to \R^{d+1}$ that sends $\x \mapsto (\x, w(\x))$. A \textbf{regular subdivision} of $P$ is precisely a subdivision $\mathcal{S}$ formed by taking the facets in the upper hull of the lifted polytope, and projecting them back to $\R^d$.

The following proposition gives an equivalent definition of regular subdivisions due to Gelfand, Kapranov, and Zelevinsky in \cite{GKZ}.

\begin{prop}\cite[Section 2.3]{Triangulations}\label{prop: regular subdiv}
  A polyhedral complex is a regular subdivision of a point configuration if and only if the cells are the domains of linearity of a convex and piecewise-linear function.
\end{prop}
 
Now we may rephrase Theorem \ref{thm:main} in the language of a completely discrete geometric problem. Let $P$ be the $(d+1)$-polyhedron in Theorem \ref{thm:main}. Intersecting the normal fan $\mathcal{N}(P)$ with the horizontal hyperplane $H=\{\y: y_{d+1} =1 \}$ yields a $d$-dimensional polyhedral complex $\mathcal{S}= \mathcal{N}(P) \cap H$. Since the support function $h_P$ is convex and piecewise linear, and it is linear in each normal cone in $\mathcal{N}(P)$, the cells in $\mathcal{S}$ are exactly the domain of linearity of the restriction of the support function $h_P(\x,1)$. By Proposition \ref{prop: regular subdiv}, $\mathcal{S}$ forms a regular subdivision of a polytope, where each vertex comes from a generating ray of $\mathcal{N}(P)$, and each cell is a slice of a normal cone in $\mathcal{N}(P)$. We can thus restate Theorem \ref{thm:main} as follows.
\begin{theorem}[Stabbing Cells in subdivision]\label{thm: main (stbbing version)}
  For every integer $d \geq 3$, there exists a regular polyhedral subdivision $S$ of a configuration of $(2d+1)(d+1)$ points in $\mathbb{R}^d$ such that there exists an interior point $\x^*$ in a cell of $S$ such that every ray emanating from $\x^*$ crosses the interior of at least $2^{d-1}$ distinct full-dimensional cells in $S$.
\end{theorem}

\begin{proof}[Proof of Corollary \ref{cor:stabbing}]
  
It is a well-known fact that any polytope can be made a simple polytope by perturbing its facets. Algebraically, this is done by slightly perturbing the right-hand side vector $\mathbf{b}$ in the inequality system $A\x\leq \mathbf{b}$ that defined the polytope. The perturbation gives us a new polytope $P'$ whose facets' outer normals remain the same, but the normal fan $\mathcal{N}(P)$ becomes a simplicial refinement of $\mathcal{N}(P_0)$: each normal cone in $\mathcal{N}(P)$ is triangulated into smaller simplicial chambers.

Since all outer normals of $P'$ still point upward, intersecting $\mathcal{N}(P')$ with $H = \{y_{d+1} =1 \}$ yields a a regular triangulation of the bounded polytope $S = \mathcal{N}(P')\cap H$. Every normal cone $N_\v$ in $\mathcal{N}(P')$ becomes a $d$-simplex $\Delta_\v$ in $S$. Since a refinement of the normal fan can only preserve or increase the number of chambers a ray crosses, by Theorem \ref{thm: main (stbbing version)}, we are done.
\end{proof}

\subsubsection{Maximizing parametric linear functions}

Alternatively, the statement of Theorem \ref{thm:main} can be framed in the terminology of parametric linear optimization. Since a ray intersecting a sequence of normal cones corresponds to tracking the sequence of maximizers under a changing linear objective function, we obtain the following equivalent formulation:

\begin{theorem}[Maximizers of linear functionals]\label{thm: main PLP Version}
  For every $d\geq 3$, there exists a $(d+1)$-dimensional polyhedron $P$ with $(2d+1)(d+1)$ facets such that for the fixed vector $\w= \e_{d+1}$ and any non-zero horizontal direction vector $\c$ (with $c_{d+1}=0$), the parametric linear function
  \[ \max_{\y\in P} (\e_{d+1} + \lambda \c)^\T \y \]
  has a sequence of $2^{d-1}$ maximizers as $\lambda$ goes to infinity.
\end{theorem}
\begin{proof}[Proof of Corollary \ref{cor:parametric optimization}]
By Proposition \ref{prop: orthogonality}, there exists a vector $\w'\in \R^{d+1}$ such that $\w'$ is orthogonal to $\c$ and the ray 
\[ \r' = \{\w' + \lambda\c \;:\; \lambda\geq 0 \} \]
is contained in the line $L = \{\w+\lambda \c) : \lambda\in \R\}$. Apply an invertible linear transformation $T$ that maps $\w'$ to $\e_{d+1}$ and $\c$ to a horizontal vector $T(\c)$. By Theorem \ref{thm: main PLP Version}, there exists a polyhedron $P$ with $(2d+1)(d+1)$ facets such that the parametric optimization problem
  \[ (\e_{d+1} + \lambda T(\c))^\T \y \]
  has a sequence of $2^{d-1}$ maximizers over $P$ as $\lambda$ goes to infinity.

Let $P' = T^\T(P)$. Since $T$ is just rotation and rescaling, the sequence of maximizer faces is preserved under the transformation, hence the parametric optimization problem
  \[ (\w'+ \lambda \c)^\T \y \]
  has a sequence of $2^{d-1}$ maximizers over $P'$ as $\lambda$ goes to infinity.
\end{proof}

\section*{AI Usage Statement}

No generative AI was used in a meaningful way for this paper at any stage of the process.


\bibliographystyle{amsplain}
\bibliography{bibliography.bib}

\end{document}